\documentclass[12pt,reqno,a4paper]{amsart}
\usepackage{amsmath,amssymb,amsfonts,amscd}
\usepackage[mathscr]{eucal}
\usepackage{hyperref}

\author{Theodore Voronov}
\address{School of Mathematics, University of Manchester,
Oxford Road, Manchester, M13 9PL, United Kingdom}
\email{theodore.voronov@manchester.ac.uk}
\title[On a non-Abelian Poincar\'{e} lemma]{On a non-Abelian Poincar\'{e} lemma}

\date{20 December 2010 (2 January 2011)}

\keywords{Maurer--Cartan equation, Lie superalgebras, differential
forms, supermanifolds, Lie algebroids, homological vector fields,
multiplicative integral, $Q$-manifolds, Quillen's superconnection}

\numberwithin{equation}{section}

\newtheorem{thm}{Theorem}[section]

\newtheorem{lm}[thm]{Lemma}
\newtheorem{prop}{Proposition}[section]
\newtheorem{cor}{Corollary}[section]

\theoremstyle{definition}

\newtheorem{ex}{Example}[section]
\newtheorem{rem}{Remark}[section]

\def\co{\colon\thinspace}

\DeclareMathOperator{\Vect}{\mathfrak{X}}

\DeclareMathOperator{\Ad}{Ad}

\newcommand{\oder}[2]{{\frac{d {#1}}{d {#2}}}}
\newcommand{\der}[2]{{\frac{\partial {#1}}{\partial {#2}}}}

\newcommand{\R}[1]{{\mathbb R}^{#1}}

\newcommand{\Z}{{\mathbb Z_{2}}}
\newcommand{\ZZ}{{\mathbb Z}}

\newcommand{\w}{{\boldsymbol{w}}}

\renewcommand{\a}{\alpha}
\renewcommand{\b}{\beta}

\def\s{\sigma}

\def\O{\Omega}
\def\D{\Delta}
\def\o{\omega}

\newcommand{\h}{\eta}

\newcommand{\x}{{\xi}}

\DeclareMathOperator{\texp}{{T}\,{exp}}

\newcommand{\texpp}{\texp\, {p_*}}
\DeclareMathOperator{\texpint}{\texp\!\!\int_{0}^{1}}
\DeclareMathOperator{\texpintt}{\texp\!\!\int_{0}^{\mathit{t}}}

\DeclareMathOperator{\Paths}{\mathbf{Path}}

\DeclareMathOperator{\Frame}{\Phi}

\newcommand{\mdi}{\texpint}
\newcommand{\mdit}{\texpintt}
\newcommand{\texpinta}[2]{\texp\!\!\int_{#1}^{#2}}

\DeclareMathOperator{\curv}{\mathbf{curv}}

\begin{document}


%



\begin{abstract}
We show that a well-known result on solutions of the Maurer--Cartan
equation extends to arbitrary (inhomogeneous) odd forms:  any such
form with values in a Lie superalgebra satisfying $d\o+\o^2=0$ is
gauge-equivalent to a constant, $$\o=gCg^{-1}-dg\,g^{-1}\,.$$ This
follows from a non-Abelian version of a chain homotopy formula
making use of multiplicative integrals. An application to Lie
algebroids and their non-linear analogs is given.

Constructions presented here generalize to an abstract setting of differential Lie superalgebras where we arrive at the statement that odd elements (not necessarily satisfying the Maurer--Cartan equation) are homotopic\,---\,in a certain particular sense\,---\,if and only if they are gauge-equivalent.
\end{abstract}
\maketitle

\section{Introduction}

It is well known that a $1$-form with values in a Lie algebra
$\mathfrak{g}$  satisfying the Maurer--Cartan
equation
\begin{equation}\label{intro_eq.mc1}
    d\o+\frac{1}{2}[\o,\o]=0
\end{equation}
possesses a `logarithmic primitive', locally or for a
simply-connected domain:
\begin{equation}\label{intro_eq.puregauge}
    \o=-dg\,g^{-1}
\end{equation}
for some $G$-valued function, where $G$ is a Lie group with the Lie
algebra $\mathfrak{g}$. Here and in the sequel we write all formulas
as if   our algebras and groups were matrix; it makes the equations
more transparent and certainly nothing prevents   from rephrasing them
in an abstract form. Therefore the main
equation~\eqref{intro_eq.mc1} will be also  written as
\begin{equation} \label{intro_eq.mc2}
    d\o+\o^2=0\,.
\end{equation}
In the physical parlance, \eqref{intro_eq.puregauge} means that the
$\mathfrak{g}$-valued  $1$-form $\o$ is a `pure gauge', i.e., the
gauge potential $\o$ is gauge-equivalent to zero.
Mathematically~\eqref{intro_eq.mc1} means that the operation
\begin{equation}\label{intro_eq.covder}
    d+\o
\end{equation}
defines a flat connection  and $g=g(x)$ in \eqref{intro_eq.puregauge} is
a choice of a parallel frame by which $d+\o$ can be reduced to the
trivial connection $d$.

On the other hand, in the Abelian case (for example, for
scalar-valued forms), the Maurer--Cartan equation becomes simply
$d\o=0$, the equation~\eqref{intro_eq.puregauge} becomes $\o=-d\ln
g=df$ for a non-zero function $g(x)=e^{-f(x)}$, and the whole
statement is just a particular case of the Poincar\'{e} lemma for
$1$-forms.

We   show that there is a non-Abelian analog of the Poincar\'{e}
lemma in full generality. Namely, instead of a \textbf{one}-form we
consider an arbitrary  \textbf{odd}  form $\o$ with values in a Lie
superalgebra $\mathfrak{g}$. `Odd' here means, in the sense of total
$\Z$-grading (parity) taking account of   parity of elements of the
Lie superalgebra. The form may very well be inhomogeneous w.r.t.
degree, i.e., be the sum of a $0$-form, a $1$-form, a $2$-form,
etc., or even be a pseudodifferential form on a supermanifold.

The statement then is as follows. If the odd form $\o$
satisfies~\eqref{intro_eq.mc1} or~\eqref{intro_eq.mc2}, then it is
gauge-equivalent to a constant: there is a $G$-valued form $g$
(necessarily even) such that
\begin{equation}\label{intro_eq.constant}
    \o=gCg^{-1}-dg\,g^{-1}\,,
\end{equation}
where $C$ is a constant odd element of the superalgebra
$\mathfrak{g}$ satisfying \mbox{$C^2=0$}. 
Here $G$ is a Lie supergroup corresponding to the
Lie superalgebra $\mathfrak{g}$. (Note that $g$ is not a function as in the classical situation, but a form. How to understand  such `$G$-valued forms' is explained in the main text.)

This is a precise analog of the statement that every closed form is
exact (in a contractible domain) plus a constant. The appearance of
a constant $C$ in~\eqref{intro_eq.constant} is a crucial
non-trivial feature, distinguishing our statement from the
well-known case when $\o$ is a $\mathfrak{g}$-valued $1$-form. The arising of a constant may seem a very innocent change, for a superficial glance, but examples show that it is not. (In fact, this odd $\mathfrak{g}$-valued constant may be a whole homological vector field defining a particular structure, see  section 4.)

From a `connections viewpoint', the meaning of our statement is that
the operation~\eqref{intro_eq.covder}, which now can be interpreted
as, say, Quillen's superconnection~\cite{quillen:superconnection85}
(it is not an ordinary connection if $\o$ is not a $1$-form), is
equivalent, by taking the conjugation with a group element $g$ (an
invertible even form taking values in $G$), to
\begin{equation*}
    d+C\,.
\end{equation*}

We discuss an application to Lie algebroids and
their non-linear analogs. No
doubt, the statement has other applications.

Our  `non-Abelian Poincar\'{e} lemma' is deduced from a more general theorem, which may
be seen as the non-Abelian analog of a homotopy formula for differential
forms. In the last section we discuss a relation with homotopy theory in a more abstract setup of differential Lie superalgebras and differential Lie supergroups.

\section{`Multiplicative direct image' and a homotopy
formula}\label{sec.multi}

Consider a (super)manifold $M$ and the direct product $M\times I$
where $I=[0,1]$.  Let $\mathfrak{g}$ be a Lie superalgebra with a
Lie supergroup $G$.

\begin{rem}
The Lie superalgebra $\mathfrak{g}$ may be finite-dimensional, but
not necessarily. Of course, for infinite-dimensional algebras,
interesting for applications, the existence of a corresponding
supergroup $G$ and of multiplicative integrals, see below, need to
be established in each concrete case.  One particular class of examples to which considerations of this paper readily apply, consists of groups (or supergroups) of transformations of a supermanifold.  As mentioned above, in the
sequel we use the notation mimicking the case of matrix algebras and
matrix groups, but everything makes sense in the abstract setting.
\end{rem}

We shall define a map
\begin{equation*}
    \O^{\text{odd}}(M\times I,\mathfrak{g})\to
    \O^{\text{even}}(M,G)
\end{equation*}
from odd $\mathfrak{g}$-valued forms on $M\times I$ to even
$G$-valued forms on $M$, corresponding to the projection $p\co M\times
I\to M$. Here a rigorous understanding of an `even form with values
in some (super)manifold', say, $N$ (in our case it is a supergroup
$G$), is a map of supermanifolds $\Pi TM\to N$ (in the case under
consideration, a map $\Pi TM\to G$).

\smallskip
{

 \small

For a matrix (super)group  $G$, a $G$-valued form $g=g(x,dx)$ on an
ordinary manifold is an expansion $g=g_0(x)+dx^ag_a(x)+\ldots$ where
the zero-order term is an invertible matrix-function such that
$g_0(x)$ belongs to $G$ for each $x$,  and the other terms are
`higher corrections'. Note however that the whole sum $g(x,dx)$, not
only $g_0(x)$, has to satisfy the equations specifying the group
manifold $G$, so there are relations for the higher terms as well.
The whole sum must be of course even in the sense of total parity.

}
\smallskip

The desired map will be called the \textit{multiplicative direct
image} or \textit{multiplicative fiber integral} and denoted
\begin{equation}\label{multi_eq.multdirimage}
  \texpp\co \O^{\text{odd}}(M\times I,\mathfrak{g})\to
   \O^{\text{even}}(M,G)\,.
\end{equation}
As the name suggests, the construction of the multiplicative
direct image goes as follows. For a given $\o\in
\O^{\text{odd}}(M\times I,\mathfrak{g})$, we consider the
decomposition
\begin{equation*}
    \o=\o_0+dt\,\o_1
\end{equation*}
where $t$ is the coordinate on $I$. The form $\o_1$ is an even
$\mathfrak{g}$-valued form on $M$ depending on $t$ as a parameter
or, alternatively, it can be regarded as a function of $t\in I$ with
values in the Lie algebra (not superalgebra!)
$\O^{\text{even}}(M,\mathfrak{g})$. It makes sense to consider the
Cauchy problem
\begin{equation}\label{multi_eq.cauchy}
    \left\{
    \begin{aligned}
        g(0)&=1\,,\\
        \oder{g(t)}{t}&=-\o_1g(t)\,
    \end{aligned} \right.
\end{equation}
(the choice of the minus sign is dictated by the geometric
tradition). The solution $g(t)$, which takes values in the group
$\O^{\text{even}}(M,G)$, will be denoted
\begin{equation}\label{multi_eq.multdirimageatt}
    \mdit\!(-\o):=g(t)\,,
\end{equation}
for $t\in[0,1]$.  We may also need the solution at time
$t = t_1$ of the   Cauchy problem with the initial value $1$ at $t=t_0$, which will be denoted $g(t_1,t_0)$\,:
\begin{equation}\label{multi_eq.multdirimagegen}
    g(t_1,t_0)=\texpinta{t_0}{t_1}(-\o)\,,
\end{equation}
for $t_0, t_1\in [0,1]$, so that $g(t)=g(t,0)$\,.
In particular, we define
\begin{equation}\label{multi_eq.multdirimageat1}
    \o \mapsto \texpp\o:= g(1,0)=\mdi\!(-\o) \in \O^{\text{even}}(M,G)
\end{equation}
to be the desired map~\eqref{multi_eq.multdirimage}.

\begin{rem} The multiplicative integrals above are standard
multiplicative integrals (cf.~\cite{gantmacher:appl},
\cite{manturov:multint90}) and if so wished they can be expressed as
limits of products of exponentials or as series of integrals of
time-ordered products. We use  the ``$\texp\!\int$'' notation  from among the
variety of existing notations. In our case the integrands are
$1$-forms on $I$ with values in a Lie algebra. The corresponding Lie
group $\O^{\text{even}}(M,G)$ is the group of the $\O(M)$-points of
the Lie supergroup $G$ (that is,   maps $\Pi TM\to G$). More detailed notations 
such as
\begin{equation*}
    \texpinta{t_0}{t^1}\!dt\,\bigl(-\o_1(x,dx,t)\bigr) \quad \text{or} \quad
    \texpinta{t_0}{t^1}\!D(t,dt)\,\bigl(-\o(x,dx,t,dt)\bigr)
\end{equation*}
are also possible. In the latter formula $D(t,dt)$ stands for the
Berezin integration element w.r.t. the variables $t,dt$ and the role of
the Berezin integration over the odd variable $dt$ is in isolating
the term $\o_1$.
\end{rem}

For an   Abelian $\mathfrak g$, the multiplicative direct image $\texpp$ is   the
composition of the ordinary direct image $p_*=\int_0^1$ taking forms
on $M\times I$ to forms on $M$ (with the same values) and the
exponential map.

We shall now establish an analog of the fiberwise Stokes formula for
the ordinary direct image
\begin{equation}\label{multi_eq.abelstokes}
    d\circ p_*+p_*\circ d= p'_*\,.
\end{equation}
Here $p'_*$ stands for the `integral over
the fiberwise boundary', so that, for a form $\s(x,dx,t,dt)$ on $M\times I$,\,  $p'_*\s=\s(x,dx,t,0)|^{t=1}_{t=0}$.

For such an analog, the first term in
equation~\eqref{multi_eq.abelstokes} should be replaced by the
multiplicative direct image followed by the Darboux differential
$g\mapsto \D (g)=-dg\,g^{-1}$, while the operator $d$ in the second
term should be replaced by the non-linear operation of `taking
curvature':
\begin{equation}
    \o\mapsto \curv\o:=d\o+\o^2\,.
\end{equation}
As we shall see,  certain other modifications  arise for  both sides and the  formula as a whole is slightly more complicated than a
naive analog of~\eqref{multi_eq.abelstokes}. (At the same time, in
hindsight all extra complications allow a geometric
interpretation and are geometrically natural.)

\begin{thm} \label{multi_thm.nonabhomotopy}
Let $\o=\o_0+dt\,\o_1\in \O^{\text{\emph{odd}}}(M\times I,
\mathfrak{g})$ be an odd $\mathfrak{g}$-valued form on the direct
product $M\times I$. Define $g(t_1,t_0)$ as above.  
The following commutation formula holds:
\begin{equation}\label{multi_eq.nonabelstokes1}
\boxed{
    \D\circ \texpp + p_*\circ \Ad g(1,t) \circ \curv = p'_* \circ \Ad g(1,t)\,,\vphantom{\int_0^1}
}
\end{equation}
where $p'_*=|^{t=1, dt=0}_{t=0, dt=0}$ is the   `integral over the fiberwise boundary'.

If we denote $g:=\texpp \o=g(1,0)$ and   $\O:=\curv\o$\,, then the   formula reads
\begin{equation}\label{multi_eq.nonabelstokes}
\boxed{
    -dg\,g^{-1} + \int_0^1\!\! g(1,t)\,\O\, g(1,t)^{-1}
    ={\o_0}|_{t=1}-g\bigl({\o_0}|_{t=0}\bigr)g^{-1}\,.
}
\end{equation}
Here $p_*=\int_0^1$ in the second term is the ordinary fiberwise
integral applied to a $\mathfrak{g}$-valued form on $M\times I$.
\end{thm}

\begin{rem} The form $\o$ is odd; obviously, the $G$-valued form $g$
is even, hence the form $-dg\,g^{-1}$ is odd. The form $d\o+\o^2$ is
even and the fiberwise integration makes an even form, odd. The term
$\o_0$ is odd. Therefore the main
equation~\eqref{multi_eq.nonabelstokes} is an equality between odd
$\mathfrak{g}$-valued forms on $M$.
\end{rem}

\begin{rem} The appearance of the conjugation with $g(1,t)$  or $g=g(1,0)$ in~\eqref{multi_eq.nonabelstokes1}, \eqref{multi_eq.nonabelstokes}  becomes  geometrically transparent
if one thinks of $g(t_1,t_0)$ as of a `parallel transport' along $\{x\}\times I$ from $t=t_0$ to  $t=t_1$ (see more in Section~\ref{sec.discussion}). The whole formula~\eqref{multi_eq.nonabelstokes} may be viewed as written at the time $t=1$.
\end{rem}

\begin{proof}[Proof of Theorem~\ref{multi_thm.nonabhomotopy}]
Equation~\eqref{multi_eq.nonabelstokes} is an explication of~\eqref{multi_eq.nonabelstokes1}, so we need to prove~\eqref{multi_eq.nonabelstokes}.
We shall deduce a formula equivalent
to~\eqref{multi_eq.nonabelstokes}, from
which~\eqref{multi_eq.nonabelstokes} will follow by conjugation:
\begin{equation}\label{multi_eq.nonabelstokes2}
    -g^{-1}dg + \int_0^1\!\! g(0,t)^{-1}\,\O\, g(0,t)
    =g^{-1}\bigl({\o_0}|_{t=1}\bigr)g-{\o_0}|_{t=0}\,.
\end{equation}
For brevity we shall use the notation $g_t=g(0,t)$. Consider the first
term in~\eqref{multi_eq.nonabelstokes2}. To obtain an expression for
it, notice that $g^{-1}dg=(g_t^{-1}d_xg_t)|_{t=1}$ and
$d_xg_t|_{t=0}=0$. By differentiating and
using~\eqref{multi_eq.cauchy} we arrive after a simplification at
\begin{equation*}
    \oder{}{t}\left(g_t^{-1}d_xg_t\right)=-g_t^{-1}d_x\o_1g_t\,.
\end{equation*}
Therefore
\begin{equation}\label{multi_eq.ginvdg}
    g^{-1}dg=\int_0^1dt
    \oder{}{t}\left(g_t^{-1}d_xg_t\right)=\int_0^1dt\left(-g_t^{-1}d_x\o_1g_t\right)\,.
\end{equation}
Consider now the curvature $\O=d\o+\o^2$. We have
\begin{equation}\label{multi_eq.curv}
    d\o+\o^2=d_x\o_0+\o_0^2+dt\Bigl(-d_x\o_1+\dot\o_0+[\o_1,\o_0]\Bigr)\,,
\end{equation}
by a direct calculation (where the dot denotes the time derivative).
Therefore
\begin{equation}\label{multi_eq.intofcurv}
    \int_0^1g_t^{-1}\O g_t=\int_0^1dt\left(-g_t^{-1}d_x\o_1g_t\right)+ \int_0^1dt\Bigl(g_t^{-1}\bigl(\dot \o_0+[\o_1,\o_0]\bigr)g_t\Bigr)\,.
\end{equation}
It follows by combining~\eqref{multi_eq.ginvdg}
and~\eqref{multi_eq.intofcurv} that
\begin{equation}\label{multi_eq.lhs}
    -g^{-1}dg+\int_0^1g_t^{-1}\O g_t=\int_0^1dt\Bigl(g_t^{-1}\bigl(\dot \o_0+[\o_1,\o_0]\bigr)g_t\Bigr)\,.
\end{equation}
It remains to identify the r.h.s. of~\eqref{multi_eq.lhs}. To this
end, consider $g_t^{-1}\o_0g_t$. By differentiating and taking into
account~\eqref{multi_eq.cauchy} we obtain, after a simplification,
that
\begin{equation}\label{multi_eq.rhs}
    \oder{}{t}\left(g_t^{-1}\o_0g_t\right)=g_t^{-1}\bigl(\dot \o_0+[\o_1,\o_0]\bigr)g_t\,.
\end{equation}
This is exactly what we are looking for. Combining
now~\eqref{multi_eq.lhs} with~\eqref{multi_eq.rhs} we arrive at
\begin{equation}\label{multi_eq.almostfinal}
    -g^{-1}dg+\int_0^1g_t^{-1}\O g_t=\int_0^1dt
    \oder{}{t}\left(g_t^{-1}\o_0g_t\right)\,,
\end{equation}
which gives~\eqref{multi_eq.nonabelstokes2} as desired, since
$g_1=g$ and $g_0=1$. To get~\eqref{multi_eq.nonabelstokes} we apply
the conjugation by the element $g=g(1,0)$, and the theorem is proved.
\end{proof}

Theorem~\ref{multi_thm.nonabhomotopy} implies an analog of the chain
homotopy formula for homotopic maps $M\to N$. Suppose $F\co M\times
I\to N$ is a homotopy between maps $f_0,f_1\co M\to N$, so that
$F(x,t)=f_t(x)$ for $t=0,1$. Consider an odd form $\o\in
\O^{\text{odd}}(N,\mathfrak{g})$ and take its pull-back $F^*\o\in
\O^{\text{odd}}(M\times I,\mathfrak{g})$. By applying
equation~\eqref{multi_eq.nonabelstokes} to $F^*\o$ and noting that
taking curvature commutes with pull-backs, we arrive at the
following statement.

\begin{cor}[Non-Abelian algebraic homotopy] \label{multi_cor.nonabhomotopy}The pull-backs along
homotopic maps $f_0, f_1\co M\to N$ of an odd $\mathfrak{g}$-valued
form $\o\in \O^{\text{\emph{odd}}}(N,\mathfrak{g})$ are related by
the formula
\begin{equation}\label{multi_eq.nonabelhomotopy}
    f_1^*\o-g (f_0^*\o) g^{-1}=-dg\,g^{-1} + \int_0^1\!\! g(1,t)\,F^*(\curv \o)\, g(1,t)^{-1}
\end{equation}
where $g=g(1,0)$ and
\begin{equation}\label{multi_eq.nonabelhomotopygt}
    g(t_1,t_0)=\texpinta{t_0}{t_1} (-F^*\o)\,.
\end{equation}
Here $F$ is a given homotopy, $F(x,t)=f_t(x)$.

In particular, for flat forms, the pull-backs along homotopic maps
are gauge-equivalent:
\begin{equation}\label{multi_eq.nonabelhomotopyflat}
    f_1^*\o=g (f_0^*\o) g^{-1}-dg\,g^{-1}\,,
\end{equation}
where `flat' means vanishing curvature: $\curv \o=d\o+\o^2=0$.  
\end{cor}

\begin{rem}\label{multi_rem.abstract}
Constructions of this section can be done in a more general abstract setup. Namely, we can replace  $\O(M,\mathfrak g)$ by an abstract differential Lie superalgebra. Suppose  $\boldsymbol{\mathfrak g}$ (we use boldface) is such a differential Lie superalgebra (shortly, a `dLie algebra'). Then the even forms on $M$ with values in $G$ are replaced, respectively, by points of a corresponding `differential Lie supergroup' (`dLie group' or   `$Q$-group') $\boldsymbol{G}$.  Odd forms on the cylinder $M\times I$ are modeled, in such an abstract setting, by the elements of the tensor product   $\Pi \boldsymbol{\mathfrak g}\otimes \O(I)=\O(I,\Pi \boldsymbol{\mathfrak g})$. The construction of the multiplicative direct image and the analog  of  Theorem~\ref{multi_thm.nonabhomotopy}   carry through without any substantial change. The multiplicative direct image should be regarded as a supermanifold map $\O(I,\Pi \boldsymbol{\mathfrak g})\to \boldsymbol{G}$. Since our motivation is mainly differential-geometric, we stick to the concrete setting of differential forms. However, the abstract formulation may be of independent interest. We briefly discuss it in the last section.
\end{rem}

\section{Non-Abelian  Poincar\'{e} lemma}

Results of the previous section can be applied for obtaining a non-Abelian version of the
Poincar\'{e} lemma. We can use
Corollary~\ref{multi_cor.nonabhomotopy} or argue directly, as
follows. Let $\o$ be an odd $\mathfrak{g}$-valued form on a
contractible supermanifold, for example on a star-shaped domain
$U\subset \R{n|m}$. Suppose it satisfies the Maurer--Cartan
equation:
\begin{equation}\label{nonabelpoinc_eq.mc}
    d\o+\o^2=0\,.
\end{equation}
Consider a contracting homotopy $H\co M\times I\to M$, for example,
the map $H\co U\times I\to U$ sending $(x,t)$ to $tx$. Take the
pull-back $H^*\o$ of $\o$ and apply to it the main
formula~\eqref{multi_eq.nonabelstokes}. We arrive at
\begin{equation*}
    -dg\,g^{-1}+0 = (H^*\o)|_{t=1, dt=0}- g\bigl((H^*\o)|_{t=0, dt=0}\bigr)g^{-1}
\end{equation*}
where
\begin{equation} \label{nonabelpoinc_eq.nabpoinc-primit}
    g=\mdi\! (-H^*\o)\,.
\end{equation}
Noticing that $(H^*\o)|_{t=1, dt=0}=\o$ and $(H^*\o)|_{t=0,
dt=0}=i^*\o$ where $i$ is the inclusion of the base point to $M$, we
get simply
\begin{equation*}
    -dg\,g^{-1} = \o- g(i^*\o) g^{-1}\,.
\end{equation*}
Here $i^*\o$ is just an odd element of the Lie superalgebra
$\mathfrak{g}$.

We have proved the following statement.

\begin{thm}[Non-Abelian Poincar\'{e} Lemma]
\label{nonabelpoinc_thm.napoinlemma} On a contractible
supermanifold, an odd $\mathfrak{g}$-valued form $\o\in
\O^{\text{\emph{odd}}}(M,\mathfrak{g})$ satisfying the
Maurer--Cartan equation~\eqref{nonabelpoinc_eq.mc} is
gauge-equivalent to a constant:
\begin{equation}\label{nonabelpoinc_eq.nabpoinc}
    \o=gCg^{-1}-dg\,g^{-1}\,,
\end{equation}
where $C\in \mathfrak{g}_{\bar 1}$ is an odd   element of the Lie
superalgebra $\mathfrak{g}$ and a `multiplicative primitive' $g\in
\O^{\text{\emph{even}}}(M,G)$ of the form $\o$ is given
by~\eqref{nonabelpoinc_eq.nabpoinc-primit}.
\end{thm}

In particular,  the formula for a `multiplicative primitive' of an
odd form satisfying the Maurer--Cartan equation on a star-shaped
domain of $\R{n|m}$, is
\begin{multline} \label{nonabelpoinc_eq.nabpoinc-primitstar}
    g=\mdi\!\! (-D(t,dt)\,\o(tx,dt\, x+t\,dx)) \\
    =\mdi\!\! \left(-dt\,x^a\der{\o}{dx^a}(tx,t\,dx)\right)\,,
\end{multline}
very similar to the classical formula for a primitive of a closed
form on $\R{n}$ or $\R{n|m}$.

Since gauge transformations preserve the Maurer--Cartan equation,
the constant element $C\in \mathfrak{g}_{\bar 1}$ must also satisfy
it. Hence the following holds.

\begin{prop} The constant  $C$ (an odd   element of the Lie
superalgebra $\mathfrak{g}$)
in~\eqref{nonabelpoinc_eq.nabpoinc} satisfies
\begin{equation}\label{nonabelpoinc_eq.homologicc}
    C^2=0 \quad \text{(or $[C,C]=0$)}\,,
\end{equation}
i.e., $C$ is a homological element of the Lie superalgebra
$\mathfrak{g}$. 
\end{prop}

If we split $\o$ and other forms as
    $\o=\o_0+\o_{+}$,
where $\o_0=\o|_M$ or more precisely $\o_0=\pi^*i^*\o=$ (here
$\pi\co\Pi TM \to M$ is  the projection and $i\co M\to \Pi TM$,  the
zero section), then equation~\eqref{nonabelpoinc_eq.nabpoinc}
becomes
\begin{align}\label{nonabelpoinc_eq.splitomeganul}
    \o_0&=g_0Cg_0^{-1}\,, \\
    \label{nonabelpoinc_eq.splitomegaplus}
    \o_{+}&=\left(gCg^{-1}\right)_{+}-dg\,g^{-1}\,.
\end{align}
In general, $\o_0$ need not be constant;
equation~\eqref{nonabelpoinc_eq.splitomeganul} implies $\o_0$ being
`covariantly constant' w.r.t. to a flat connection:
$d\o_0+[\theta,\o_0]=0$ where $\theta=-dg_0\,g_0^{-1}$.

The constant $C$ is not unique. Our proof of
Theorem~\ref{nonabelpoinc_thm.napoinlemma} gives $C$ as the value of
$\o$ at the base point $x_0$ of a contractible manifold $M$. In
fact,  it can be replaced by  any constant conjugate to the value of
$\o$ at  $x_0$ w.r.t. the adjoint action of $G$. It is easy to
deduce the following statement.

\begin{cor} On a contractible manifold, two odd forms satisfying the
Maurer--Cartan equation are gauge-equivalent if and only if their
values at the base point are conjugate. 
\end{cor}

Hence there is a one-to-one correspondence between the gauge
equivalence classes of `flat forms' (or flat $G$-superconnections)
on a contractible manifold and the $G$-adjoint orbits of the
homological elements in $\mathfrak{g}$. More properly this should be
stated not as a bijection of sets, but as an isomorphism of the
corresponding functors so as to allow arbitrary families.

It is possible to have non-trivial gauge equivalences between
different constants, as well as self-equivalences. So the
multiplicative primitive $g$ in~\eqref{nonabelpoinc_eq.nabpoinc} is
not defined uniquely even for a fixed constant $C$. If
\begin{equation}\label{nonabelpoinc_eq.nabpoinc2}
    \o=hC'h^{-1}-dh\,h^{-1}\,,
\end{equation}
for some  other  $C'\in \mathfrak{g}_{\bar 1}$ and $h\in
\O^{\text{{even}}}(M,G)$, then there is a relation
\begin{equation}\label{nonabelpoinc_eq.nabpoinc3}
    C=k  C' k^{-1} -dk\,k^{-1}
\end{equation}
where $k=g^{-1}h$. The form $k\in \O^{\text{even}}(M,G)$   satisfies
\begin{equation}\label{nonabelpoinc_eq.gaugeofconst}
    [C,dk\,k^{-1}]+(dk\,k^{-1})^2=0\,.
\end{equation}
Suppose there is a gauge self-equivalence
\begin{equation}\label{nonabelpoinc_eq.selfgaugeofconst}
    C=gCg^{-1}-dg\,g^{-1}\,,
\end{equation}
for a homological element $C$.
Equation~\eqref{nonabelpoinc_eq.selfgaugeofconst} can be re-written
as
\begin{equation}\label{nonabelpoinc_eq.selfgaugeofconst2}
   dg=- Cg+gC \,,
\end{equation}
with the clear intuitive meaning of the `frame' $g$ being parallel
w.r.t. the (constant) `flat connection' $C$.

In the classical (Abelian) situation, the usual formulation of the
Poincar\'{e} lemma (``every closed form is exact'') is for forms of  fixed positive degree. Forms of
degree zero and  constants   appear as an exception unless
inhomogeneous forms are treated. In the non-Abelian case,
inhomogeneous $\mathfrak{g}$-valued forms do appear naturally. As
soon as we consider them, we are forced to consider constants
arising in~\eqref{nonabelpoinc_eq.nabpoinc}. These constants play an
important role in examples as we shall see  now.

\section{Examples}

In this section we  change our notation and denote by $L$ the Lie
superalgebra  denoted previously by $\mathfrak{g}$. (We shall need
$\mathfrak{g}$ for a different object.)

In the  following examples, the Lie superalgebra $L$ will be the
algebra of vector fields $\Vect(N)$ on a (super)manifold $N$. In
particular, $N$ can be   a vector space; then the Lie superalgebra
$L=\Vect(N)$ carries an extra $\ZZ$-grading corresponding to  degree
in the linear coordinates on $N$. It should not be confused with
parity. More generally, $N$ can be an arbitrary graded manifold,
i.e., a supermanifold with an extra $\ZZ$-grading in the structure
sheaf, independent of parity in general~\cite{tv:graded}. We refer
to such a $\ZZ$-grading as \emph{weight} and denote it $\w$.

\emph{All the considerations in the previous sections remain valid
in the graded case.}

Indeed,  suppose that the Lie superalgebra $L$  is $\ZZ$-graded.
Total weight of the elements of $\O(M,L)$ is  the sum of ordinary
degree of forms on $M$ and weight on $L$. For the operator $d+\o$ to
be homogeneous, we assume that the weight $\w(\o)$ of the form $\o$
equals $+1$ (this does not mean that $\o$ is a $1$-form!). Recalling
the  construction of the multiplicative direct image on $M\times I$
for $\o=\o_0+dt\,\o_1$ in Section~\ref{sec.multi}, we see that
$\w(\o_1)=0$. Hence the multiplicative integral there and all the
gauge transformations defined with its help have weight zero, i.e.,
they all preserve weights.

Now let us turn to examples. We use the notions of Lie algebroid theory, for which the standard and encyclopedic source is Mackenzie's book~\cite{mackenzie:book2005} (see also~\cite{moerdijk:mrcun-book}).

\begin{ex} Consider the \textbf{Atiyah algebroid} of a principal $G$-bundle
$P\to M$, which is a transitive Lie algebroid over $M$.
See~\cite{mackenzie:book2005}. (Here $G$ is a Lie group, nothing to
do with what it was in the previous sections.) It is defined as
$TP/G$ and it inherits the structure  of a vector bundle over $M$.
There is an epimorphism of vector bundles $TP/G\to TM$, which is the
anchor in the Lie algebroid structure. Let us reverse parity in the
fibers and consider the supermanifold $\Pi TP/G$. We can consider it
as a fiber bundle over $\Pi TM$,
\begin{equation}\label{exam.eq.atiyah}
    \Pi TP/G \to \Pi TM\,.
\end{equation}
The standard fiber of~\eqref{exam.eq.atiyah} can be identified with
$\Pi \mathfrak{g}$, where $\mathfrak{g}$ is the Lie algebra of $G$.
The transition functions have the form
\begin{equation}\label{exam.eq.atiyahtransitions}
    \xi_{\a}=g_{\a\b}\xi_{\b}g_{\a\b}^{-1}-dg_{\a\b}\,g_{\a\b}^{-1}\,,
\end{equation}
if $g_{\a\b}$ is the cocycle defining the principal bundle $P\to M$.
They are linear in $\xi, dx$, but affine in $\xi$ alone. Here $\xi$,
or $\xi_{\a}$ in a particular local trivialization, belongs to $\Pi
\mathfrak{g}$.  Hence~\eqref{exam.eq.atiyah} is an affine bundle
over $\Pi TM$. The Lie algebroid structure of the Atiyah algebroid
is encoded in the following homological vector field on $\Pi TP/G$:
\begin{equation} \label{exam.eq.atiyahQ}
    Q=dx^a\der{}{x^a}+\frac{1}{2}\xi^i\xi^jC_{ji}^k\der{}{\x^k}\,.
\end{equation}
Here $x^a$ are local  coordinates on $M$ and  $\x^i$ are linear
coordinates on $\Pi \mathfrak{g}$. The tensor $C_{ij}^k$ gives the
structure constants of the Lie algebra $\mathfrak{g}$. One may
consider the second term in~\eqref{exam.eq.atiyahQ} as a constant
element $C$ of the Lie superalgebra $L=\Vect(\Pi \mathfrak{g})$ of
the vector fields on the supermanifold $\Pi \mathfrak{g}$, so $Q$
has the appearance
\begin{equation} \label{exam.eq.atiyahQ2}
    Q=d+C \quad \text{where} \quad C\in \Vect(\Pi
\mathfrak{g})\,.
\end{equation}
Note that $C$ is odd and has weight $1$. It satisfies $C^2=0$. In a
different language, $C$ is the differential of the standard cochain
complex of $\mathfrak{g}$ usually denoted as $\delta$ (the
Chevalley--Eilenberg or Cartan differential). It is a remarkable
fact, directly verifiable, that the
decomposition~\eqref{exam.eq.atiyahQ2} survives transformations of
the form~\eqref{exam.eq.atiyahtransitions}.
\end{ex}

\begin{ex} Let $E\to
M$ be now an arbitrary transitive Lie algebroid,
see~\cite{mackenzie:book2005}.  The anchor map $E\to TM$ is
epimorphic, so we can again consider it as a fiber bundle. One can
check that it is an affine bundle. Consider the vector bundles over
$M$ with   reversed parity in the fibers. We have the affine bundle
\begin{equation}\label{exam.eq.transitive}
    \Pi E\to \Pi TM\,.
\end{equation}
If we denote   coordinates in the fiber
of~\eqref{exam.eq.transitive} by $\x^i$, and local coordinates on
$M$ by $x^a$ as above, so that on $\Pi TM$ the coordinates are
$x^a,dx^a,\x^i$, the changes of coordinates are
\begin{equation}\label{exam.eq.transitivecoorchanges}
    \left\{
    \begin{aligned}
        x^a&=x^a(x')\,,\\
        dx^a&=dx^{a'}\der{x^a}{x^{a'}}\,,\\
        \x^i&=dx^{a'}T^i_{a'}(x')+\x^{i'}T^i_{i'}(x')\,,
    \end{aligned} \right.
\end{equation}
with some matrices $T^i_{a'}(x')$ and $T^i_{i'}(x')$. The
homological vector field $Q$ on $\Pi E$ defining the structure of a
Lie algebroid over $M$  has the form
\begin{equation}\label{exam.eq.transitiveQ}
    Q=d+\frac{1}{2}\Bigl(\xi^i\xi^jQ_{ji}^k(x)+2\x^idx^aQ_{ai}^k(x)+dx^adx^bQ_{ba}^k(x)\Bigr)\der{}{\x^k}\,
\end{equation}
in local coordinates, with $d=dx^a\partial/\partial {x^a}$. Or:
\begin{equation}\label{exam.eq.transitiveQ2}
    Q=d+\o\,,
\end{equation}
where
\begin{equation}\label{exam.eq.transitiveomega}
    \o=\frac{1}{2}\Bigl(\xi^i\xi^jQ_{ji}^k(x)+2\x^idx^aQ_{ai}^k(x)+dx^adx^bQ_{ba}^k(x)\Bigr)\der{}{\x^k}
\end{equation}
can be regarded as a (local) form on $M$ with values in the Lie
superalgebra $L=\Vect (\Pi V)$. Here $V$ is the standard fiber of
$E\to TM$. Note that the form $\o$ is odd and inhomogeneous w.r.t.
ordinary degree. However, it  is homogeneous of weight $+1$ if
weight is counted as the sum of form degree and grading of vector
fields on $\Pi V$. Claim: the
form~\eqref{exam.eq.transitiveomega} is gauge-equivalent to a
constant. This amounts to saying that by an appropriate affine
transformation of the coordinates $\x$,
\begin{equation*}
    \x=\h\cdot A(x)+\beta
\end{equation*}
where the matrix $A(x)$ depends only on $x$ and $\beta$ is a
$1$-form, one can transform the vector field $Q$ on $\Pi E$ to the
form
\begin{equation}\label{exam.eq.transitiveQ3}
    Q=d+C  \quad \text{where} \quad C=\frac{1}{2}\h^i\h^jC_{ji}^k\der{}{\h^k}\in \Vect(\Pi
V)\,,
\end{equation}
with constant $C_{ij}^k$. 
This follows by a direct application of
Theorem~\ref{nonabelpoinc_thm.napoinlemma}. That the transformation
is affine, follows from the preservation of weights and can be seen
from an explicit construction of the multiplicative integrals. In
particular, $C^2=0$, therefore the vector space $V$ is endowed with
a structure of a Lie algebra. This  encapsulates a substantial part
of the transitive Lie algebroid theory~\cite[Ch.
8]{mackenzie:book2005}.
\end{ex}

The previous example prompts an immediate analogy. A few words
should be said before we pass to it.

Recall that a \emph{$Q$-manifold} is an arbitrary supermanifold
endowed with a homological vector field. $Q$-manifolds should be
regarded, together with Poisson manifolds and Schouten (odd Poisson)
manifolds, as one of the three equally important non-linear
generalizations of Lie algebras\footnote{A Lie algebra multiplication of the elements of $\mathfrak g$ has
precisely three geometric manifestations: as a linear Poisson
bracket on $\mathfrak{g}^*$; as a linear Schouten bracket on
$\Pi\mathfrak{g}^*$; and as a quadratic homological vector field on
$\Pi\mathfrak{g}$.}. Vaintrob noted~\cite{vaintrob:algebroids} that
a Lie algebroid structure on a vector bundle $E$ is very efficiently
described by a homological vector field  on the parity-reversed
total space $\Pi E$. This approach was further developed in
particular in~\cite{roytenberg:thesis}, \cite{tv:graded},
\cite{roytenberg:graded}, and \cite{tv:mack}. This is what we used
above. (We use the notation of~\cite{tv:graded}.) For a Lie
algebroid the corresponding field $Q$ has weight $+1$ w.r.t. the
linear coordinates in the fibers. Dropping this restriction leads to
objects like strongly-homotopy Lie (or $L_{\infty}$) algebras and
algebroids. On the other hand, constructions such as the `cotangent
construction of the Drinfeld double' for Lie
bialgebroids~\cite{roytenberg:thesis} (compare
also~\cite{mackenzie:drinfeld}) take us out of the world of vector
bundles. Prompted by that, it is proper to replace a lacking linear
structure by weight and consider graded manifolds instead of vector
bundles~\cite{tv:graded}. In particular, a non-negatively graded
$Q$-manifold with $\w(Q)=+1$ is the closest analog of a Lie
algebroid and should be regarded as the `non-linear version' of
such.

\begin{ex}[``Non-linear transitive Lie algebroids''] \label{exam.ex.transnonlin}
Consider a fiber bundle
\begin{equation}\label{exam.eq.nonlineartransitive}
    p\co E\to \Pi TM
\end{equation}
in the category of graded manifolds. Here $E$ is a non-negatively
graded manifold (in particular, a supermanifold)\footnote{Comparing
with the previous example, $E$ now stands for what was $\Pi E$.},
$M$ is an ordinary supermanifold, with trivial grading, and   $\Pi
TM$ is considered with the standard vector bundle grading. (More
pedantically it should be denoted $\Pi T[1]M$.) Suppose there is a
homological vector field $Q\in \Vect(E)$ of weight $+1$ such that
$Q$ on $E$ and $d$ on $\Pi TM$ are $p$-related. Then in suitable
local coordinates $Q$ has the form
\begin{equation}\label{exam.eq.nonlineartransitiveQ}
    Q=dx^a\der{}{x^a}+Q^i(x,dx,y)\der{}{y^i}\,,
\end{equation}
where $x^a$ are local coordinates on $M$, $dx^a$ are their
differentials, and $y^i$ are coordinates in the fiber over $\Pi TM$.
The  changes of coordinates on $E$ have the form
\begin{equation}\label{exam.eq.nonlineartransitivecoorchanges}
    \left\{
    \begin{aligned}
        x^a&=x^a(x')\,,\\
        dx^a&=dx^{a'}\der{x^a}{x^{a'}}\,,\\
        y^i&=y^i(x',dx',y')\,,
    \end{aligned} \right.
\end{equation}
The appearance of $d$ as the first term
in~\eqref{exam.eq.nonlineartransitiveQ} follows from the condition
that $Q$ and $d$ are $p$-related. We have $\w(x^a)=0$, $\w(dx^a)=1$,
and $\w(y^i)=w^i>0$ are some positive integers.  The coordinate
transformations~\eqref{exam.eq.nonlineartransitivecoorchanges} are
homogeneous polynomials in $dx^{a'},y^{i'}$. The coefficients
$Q^i(x,dx,y)$ in~\eqref{exam.eq.nonlineartransitiveQ} have weights
$w^i+1$ respectively. Compare with~\eqref{exam.eq.transitiveQ}. The
whole structure consisting of the fiber bundle $E\to \Pi TM$ and the
field $Q$ with these properties should be regarded as a
\textbf{transitive non-linear Lie algebroid}. The bundle projection
$p\co E\to \Pi TM$ plays the role of the anchor. A construction of
this type with an extra restriction\,---\,see below\,---\,was put
forward in~\cite{strobl:talks, kotov:strobl2007}. They introduced
$Q$-bundles $p\co E\to B$ where both $E$ and $B$ are $Q$-manifolds
and the projection is a $Q$-morphism (the respective vector fields
are $p$-related) with a condition of `local triviality' in the
following strong sense. The standard fiber $F$ is a $Q$-manifold and
there is a bundle atlas where in local trivializations the
homological vector field $Q_E$ on $E$ takes the form of the sum
\begin{equation}\label{exam.eq.stroblgauge}
    Q_E=Q_B+Q_F
\end{equation}
where $Q_B$ and $Q_F$ are the homological vector fields on $B$ and
$F$; the transition functions are such that they respect the
decomposition~\eqref{exam.eq.stroblgauge}\footnote{There are more
conditions to be mentioned: the total space and the base are graded
manifolds and the homological fields have weights $+1$\,---\,in
fact, in~\cite{strobl:talks, kotov:strobl2007} it is assumed that
grading induces parity, which is absolutely unnecessary;  and some
Lie group of transformations preserving~\eqref{exam.eq.stroblgauge}
is   fixed  as a part of  data.}. The main case is when the
$Q$-manifold $B$ equals $\Pi TM$ for some $M$. So it is a transitive
non-linear Lie algebroid in our sense subject to the condition of
the existence of a bundle atlas for $E\to \Pi TM$ with the described
properties. We shall refer to that, in particular to
equation~\eqref{exam.eq.stroblgauge}, as to a \textbf{``Kotov--Strobl gauge''}   . Our
Theorem~\ref{nonabelpoinc_thm.napoinlemma} therefore implies that
\emph{every transitive non-linear Lie algebroid admits a Kotov--Strobl
gauge}. That means that it is possible to transform coordinates
$y^i$ to $z^j$ by a homogeneous transformation depending on $x^a,
dx^a$ as parameters in each bundle chart so that the vector
field~\eqref{exam.eq.nonlineartransitiveQ} becomes
\begin{equation}\label{exam.eq.nonlineartransitiveQguaged}
    Q=dx^a\der{}{x^a}+Q^i(z)\der{}{z^i}\,
\end{equation}
(no dependence on $x^a,dx^a$ in the second term). The graded Lie
superalgebra $L$ here is the algebra of all graded vector fields
$\Vect(F)$ on the standard fiber   $F$. In particular, we conclude
that constructions of~\cite{kotov:strobl2007} are valid for
arbitrary transitive non-linear Lie algebroids, without extra
restrictions.
\end{ex}

An appearance of  fiber bundles of the form  $E\to \Pi TM$ may look
artificial. It  helps to put it into a broader perspective as
follows. Consider an arbitrary non-negatively graded $Q$-manifold
$E$ where $\w(Q)=+1$. As noted before, it should be regarded as a
general \textbf{non-linear Lie algebroid}. At the first glance,
there is no fiber bundle there. Recall however that any
non-negatively graded manifold $E$ gives rise to a finite tower of
fibrations~\cite{tv:graded}
\begin{equation}\label{exam.eq.tower}
   E= E_N\to E_{N-1}\to \ldots \to E_2\to E_1\to E_0\,,
\end{equation}
where functions on $E_0$ have weight $0$, $E_1\to E_0$ is a vector
bundle, and $E_{k+1}\to E_k$ for higher $k$ are affine bundles. This
can be assembled into a single fiber bundle
\begin{equation}\label{exam.eq.bundle}
    E\to M\,,
\end{equation}
where $M=E_0$, the coordinates on the standard fiber have positive
weights and the transition functions are homogeneous polynomials.
The supermanifold $M$ is embedded into $E$ as a `zero section'. The
homological vector field $Q\in \Vect(E)$ in bundle coordinates has
the form
\begin{equation}\label{exam.eq.homfieldgeneral}
    Q=Q^a(x,y)\der{}{x^a}+Q^i(x,y)\der{}{y^i}\,,
\end{equation}
where $x^a$ are local coordinates on the base $M$, and $y^i$ are
coordinates on the fiber. We define a map
\begin{equation}\label{exam.eq.nonlinearanchor}
    a\co E\to \Pi TM
\end{equation}
by the formula $a^*(x^a)=x^a, a^*(dx^a)=Q^a(x,y)$ and call it the
\emph{anchor} for $E$. (Well-defined because we have a bundle, so
the transformation of the coordinates is of the form $x^a=x^a(x'),
y^i=y^i(x',y')$.) It is a bundle map over $M$, and a $Q$-morphism
(by a direct check using $Q^2=0$). In the  case of ordinary Lie
algebroids it coincides with the usual anchor after the parity
reversion. If the anchor $E\to \Pi TM$ is a surjective submersion,
we call the non-linear Lie algebroid $E$ \emph{transitive}, as in
the ordinary case~\cite{mackenzie:book2005}, and so we arrive at the
setup of Example~\ref{exam.ex.transnonlin}.

\begin{rem}
A tautological `anchor'  makes sense for any $Q$-manifold $M$
without assumption of grading: it is the vector field $Q$ itself
considered as a map $M\to \Pi TM$, which is a $Q$-morphism. Of
course  such an  `anchor' carries information about all the
algebraic structures contained in the vector field $Q$, so the name
becomes a bit arbitrary.
\end{rem}

More about the `non-linear Lie algebroids' and algebraic structure arising from them can be found in our new paper~\cite{tv:qman}.

\section{Discussion}\label{sec.discussion}

We have studied inhomogeneous or pseudodifferential odd forms on a supermanifold $M$
taking values in a Lie superalgebra $\mathfrak{g}$ and showed that
if they satisfy the Maurer--Cartan (zero curvature) equation $d\o+\o^2=0$, then, on a
contractible supermanifold, they are gauge-equivalent to constants:
\begin{equation}\label{dis.eq.gaugetoconst}
    \o=gCg^{-1}-dg\,g^{-1}\,,
\end{equation}
where $g$ is an even $G$-valued form and $C\in \mathfrak{g}_{\bar
1}$ such that $C^2=0$ (Theorem~\ref{nonabelpoinc_thm.napoinlemma}). The statement
follows from   more general claims (Theorem~\ref{multi_thm.nonabhomotopy} and
Corollary~\ref{multi_cor.nonabhomotopy}), which are a non-Abelian analog for
$\mathfrak{g}$-valued pseudodifferential forms of   the standard chain  homotopy construction. For applications, a presence of an extra  $\ZZ$-grading different from the naive form degree may be important; all statements hold true with such an extra grading taken into account. Our statements are
particularly useful in examples related with Lie algebroids and
their generalizations, but there is little doubt about their broader
significance.

Let us emphasize that in the existing literature, especially, in physics (gauge theory), ``non-Abelian Poincar\'{e} lemma'' has been almost exclusively understood as the familiar statement about Lie algebra valued $1$-forms.
Our usage is different, and this should not lead to confusion. The main novelty  of the present work is that  we study  inhomogeneous or  pseudodifferential odd forms instead of $1$-forms. When the coefficients belong to a Lie superalgebra, this  results in mixing terms of different form degrees. That  has not been done previously and we are led naturally to that by applications such as to Lie algebroids.  The  appearance of an odd constant $C$ and of  gauge transformation generated by a $G$-valued form (instead of a function) are two distinctive features of our formula~\eqref{dis.eq.gaugetoconst}.

\begin{rem}
An exception to the typical usage of ``non-Abelian Poin\-car\'{e} lemma'' as the statement about $1$-forms is paper~\cite{asada:1986} concerned with  a  `non-Abelian de Rham theory' regarded as the theory of certain cohomological sets constructed from matrix-valued $p$-forms. There is no relation
between that and our work. Shortly,    in~\cite{asada:1986},  there are a `$1$-dimensional non-Abelian Poincar\'{e} lemma' meaning the familiar statement for $1$-forms; a `$2$-dimensional non-Abelian Poincar\'{e} lemma' meaning a condition of local solubility of $d\theta +\theta^2=\Theta$, $\theta$ and $\Theta$ being matrix-valued $1$- and $2$-forms respectively; and a `$3$-dimensional non-Abelian Poincar\'{e} lemma' meaning a condition of local solubility of $d\Theta +[\theta,\Theta]=\Psi$ for a $3$-form $\Psi$.
\end{rem}

On the other hand, the results of this paper are related with homotopy theory considered from the viewpoint of algebraic deformation theory.

After the first version of this paper was finished and circulated, we learned about a ``secret'' work
by {Schles\-singer} and Stasheff~\cite{schless:stasheff-secret} (still an unpublished manuscript)  and a   preprint of Chuang and Lazarev (now published~\cite{chuang:lazarev2008}). Stasheff pointed out to us that certain formulas  such as our equation~\eqref{multi_eq.nonabelhomotopyflat} in
Corollary~\ref{multi_cor.nonabhomotopy} reminded him of some statements in~\cite{schless:stasheff-secret}  and kindly shared his text with us. To explain the relation, it is convenient to use a reformulation of
the  Main Homotopy Theorem of {Schles\-singer}--Stasheff~\cite{schless:stasheff-secret} given by Chuang and Lazarev~\cite{chuang:lazarev2008}.  The theorem concerns Maurer--Cartan elements in a nilpotent (or pro-nilpotent) dg Lie algebra. The statement is that \emph{any two such elements are gauge-equivalent if and only if they are homotopic} (the exact notion of homotopy is explained below). Apart from  considering dg algebras, not superalgebras, and imposing the (pro-)nilpotence condition, the statement is   close to an abstract algebraic version of the zero curvature case of our Corollary~\ref{multi_cor.nonabhomotopy}. The method of the proof used in~\cite{chuang:lazarev2008}\,---\,formal path-ordered integrals with values in a completion of the universal enveloping
algebra (here the pro-nilpotence is used)\,---\,is also close to ours, but is more algebraic in spirit. So in spite of a really different language, there is an overlap between some of our results and those of~\cite{schless:stasheff-secret}, \cite{chuang:lazarev2008}; a comparison can be presented as follows.

Because we are motivated by  geometric applications, we deal mainly with differential forms, while {Schles\-singer}--Stasheff and Chuang--Lazarev work in the abstract setting of   dg Lie algebras.  They consider only Maurer--Cartan elements. For us, forms satisfying the Maurer--Cartan equation  are, surely,  of great importance, but we start from establishing general formulas such as `non-Abelian chain homotopy' for the arbitrary case.

We have already mentioned at the end of section 2 (see Remark~\ref{multi_rem.abstract}) that  results of that section extend to  a more abstract framework. This was not our original goal, but may be of independent interest and allows, in particular, further comparison with  ~\cite{schless:stasheff-secret} and \cite{chuang:lazarev2008}.
We shall give more details now. We shall  obtain  a generalization of both our Theorem~\ref{multi_thm.nonabhomotopy} and the `Main Homotopy Theorem'  quoted above.

Consider a differential Lie superalgebra, i.e., a Lie superalgebra $\boldsymbol{\mathfrak{g}}$ endowed with an odd derivation $d$ such that $d^2=0$. A corresponding Lie supergroup $\boldsymbol{G}$ has a $Q$-group  structure, i.e., it is endowed with a homological vector field $Q$ such that
\begin{equation*}
    Q(g_1g_2)=Q(g_1)g_2+g_1Q(g_2)\,.
\end{equation*}
(In particular, $Q(1)=0$ and $Q(g^{-1})=-g^{-1}Q(g)g^{-1}$.) In the main text, both $d$ and $Q$ are given by the de Rham differential. It is convenient, if no confusion is possible, to continue to use the notation $d$ for both operations in the abstract setup as well. So from now on we replace $Q$ by $d$ and write $dg$ for $d(g)$.  Consider odd elements of $\boldsymbol{\mathfrak{g}}$. They should be properly regarded as points of the supermanifold $\Pi \boldsymbol{\mathfrak{g}}$. The notion of gauge-equivalence carries over verbatim to them:
\begin{equation*}
    \o'=g\o g^{-1}-dg\,g^{-1}\,.
\end{equation*}
Here $g$ should be regarded as a point of the supergroup $\boldsymbol{G}$. The Darboux differential on $\boldsymbol{G}$ is defined as $\Delta(g)=-dg \,g^{-1}$. It is a supermanifold map $\boldsymbol{G}\to \Pi \boldsymbol{\mathfrak{g}}$. We can speak about the curvature $\curv \o=d\o+\o^2$ (where $\o^2=[\o,\o]/2$, as usual)  of an element $\o\in \Pi \boldsymbol{\mathfrak{g}}$. (Upon natural identifications, $\curv \o$ is exactly the value of the homological field defining the dLie algebra structure of $\boldsymbol{\mathfrak{g}}$ at the argument $\o$.) Consider forms on the unit segment $I$ with values in $\boldsymbol{\mathfrak{g}}$. They make a dLie algebra w.r.t. the bracket extended  from $\boldsymbol{\mathfrak{g}}$ by $\O(I)$-linearity and the operator $D$,
\begin{equation*}
    D\o=d\o+dt\der{\o}{t}\,
\end{equation*}
as the differential, where the first term  is the differential on $\boldsymbol{\mathfrak{g}}$ applied `pointwise'. (This models $\mathfrak{g}$-valued forms on the cylinder $M\times I$.) We can define now the multiplicative integral
\begin{equation*}
    g(t_1,t_0)=\texpinta{t_0}{t_1}(-\o)\,,
\end{equation*}
of an odd element of $\O(I, \boldsymbol{\mathfrak{g}})$. It takes values in the supergroup $\boldsymbol{G}$. We denote by $g$ without arguments the element $g(1,0)$. 
\begin{lm} \label{dis.lemma}The following identity holds:
\begin{equation*}
    \D \Bigl(\texpint(-\o)\Bigr) + \int_0^1\!\!\! \Ad g(1,t) \,\bigl(\curv \o\bigr) = \o_0(1)-\Ad g \,\o_0(0)
\end{equation*}
or
\begin{equation*}
    -dg\,g^{-1} + \int_0^1\!\!\!  g(1,t) \,\bigl(\curv \o\bigr)g(1,t)^{-1} = \o_0(1) - g \,\o_0(0)\,g^{-1}\,.
\end{equation*}
\end{lm}
\begin{proof} This is  an  analog of Theorem~\ref{multi_thm.nonabhomotopy} and its  proof can be repeated verbatim.
\end{proof}
If we elaborate the curvature appearing under the integral above, we obtain the following. For an odd  $\o=\o_0+dt\o_1$,
\begin{align*}
    \curv \o &= d\o_0+\o_0^2 + dt \left(-d\o_1+\der{\o_0}{t}-[\o_0,\o_1]\right) \\
             & = \curv\o_0+ dt \left(-d\o_1+\der{\o_0}{t}-[\o_0,\o_1]\right)\,
\end{align*}
(compare equation~\eqref{multi_eq.curv}). Clearly, the input to the integral is given only by the second term (that contains $dt$). If the coefficient vanishes, the integral vanishes from the formula.

Now we can define a notion of `homotopy' for arbitrary odd elements of a dLie algebra $\boldsymbol{\mathfrak g}$, which extends the notion   used in~\cite{chuang:lazarev2008}  for the Maurer--Cartan elements.

We say that two odd elements, $\o^0$ and $\o^1$, in $\boldsymbol{\mathfrak g}$ are \emph{homotopic} if there is an odd element $\o\in \O(I,\boldsymbol{\mathfrak g})$ such that:

1) $\o$ specializes to $\o^0$ and $\o^1$ for $t=0$ and $t=1$ (and $dt=0$), and

2) the curvature of $\o$ does not contain $dt$.

\noindent
The element $\o$ is called a \emph{homotopy} between $\o^0$ and $\o^1$.

(It will be shown  later that for the case of zero curvature this reproduces the notion of homotopy for the Maurer--Cartan elements.)

\begin{ex} Let $\boldsymbol{\mathfrak g}=\O(M)$ for some manifold $M$, with the zero bracket and the de Rham $d$. By an easy exercise one can  see  that   two forms $\o^0, \o^1\in \O(M)$ (parity is not essential here) are homotopic in the above sense if and only if $\o^1-\o^0=d\sigma$ for some $\sigma\in \O(M)$. So in this toy case, `homotopy' it is just cohomology relation for  not necessarily closed forms.
\end{ex}

Now the main statement.

\begin{thm} \label{theorem}
Let $\boldsymbol{\mathfrak g}$ be a dLie algebra and $\boldsymbol{G}$ be a corresponding dLie group. Let $\boldsymbol{G}$ be connected. Then two odd elements $\o^0$ and $\o^1$ in $\boldsymbol{\mathfrak g}$ are homotopic if and only if they are gauge-equivalent.
\end{thm}
\begin{proof} Suppose $\o^0$ and $\o^1$  are homotopic. Then there is an odd element $\o=\o_0+dt\, \o_1\in \O(I,\boldsymbol{\mathfrak g})$ such that $\o_0(0)=\o^0$, $\o_0(1)=\o^1$  and $\curv \o$ does not contain $dt$. Then, by Lemma~\ref{dis.lemma},
\begin{equation*}
    -dg\,g^{-1}    = \o_0(1) - g \,\o_0(0)\,g^{-1}
\end{equation*}
(there is no second term at the LHS because  the integral  vanishes),
or
\begin{equation*}
   \o^1= g \,\o^0\,g^{-1} -dg\,g^{-1} \,,
\end{equation*}
i.e., $\o^0$ and $\o^1$  are gauge-equivalent.

Conversely, suppose $\o^0$ and $\o^1$  are gauge-equivalent, i.e., $\o^1= g \,\o^0\,g^{-1} -dg\,g^{-1}$,
for some $g\in \boldsymbol{G}$. Because we have assumed that $\boldsymbol{G}$ is connected, we can consider a path $h=h(t)$ such that $h(0)=1$ and $h(1)=g$. Define $\o=\o_0+dt\, \o_1\in \O(I,\boldsymbol{\mathfrak g})$ by setting
\begin{equation*}
    \o_0:=h \,\o^0\,h^{-1} -dh\,h^{-1}
\end{equation*}
(so that we indeed have $\o_0(0)=\o^0$ and $\o_0(1)=\o^1$) and
\begin{equation*}
    \o_1:=- \dot h \,h^{-1}\,.
\end{equation*}
By a direct calculation, we can see that the equation
\begin{equation*}
    \der{\o_0}{t}=-d\o_1+[\o_0,\o_1]\,,
\end{equation*}
is satisfied. Therefore $\o$ is a desired homotopy between $\o^0$ and $\o^1$.
\end{proof}

Theorem~\ref{theorem} is a generalization of the Main Homotopy Theorem  quoted above,  which is Theorem~4.4 in~\cite{chuang:lazarev2008}.
\begin{rem} \label{dis.rem.sulhom}
The notion of homotopy  for the Maurer--Cartan elements used in\cite{chuang:lazarev2008} (`Sullivan homotopy') is as follows: two Maurer--Cartan elements $\o^0$ and $\o^1$ in $\boldsymbol{\mathfrak g}$ are \emph{homotopic} if there is a Maurer--Cartan element $\o\in \O(I,\boldsymbol{\mathfrak g})$   specializing to $\o^0$ and $\o^1$ resp. for $t=0$ and $t=1$. If we apply our definition instead, it is not known a priori that a homotopy $\o$ is Maurer--Cartan. However, it automatically is, if only $\o^0$ and $\o^1$ are Maurer--Cartan, due to the following observation. We know that $\curv \o$ not containing $dt$  translates into the equation
\begin{equation*}
    \der{\o_0}{t}=-d\o_1+[\o_0,\o_1]\,.
\end{equation*}
Consider the time derivative of $\curv \o_0$. The above equation implies
\begin{equation*}
    \der{\,\curv\o_0}{t}=-[\o_1,\curv\o_0]\,.
\end{equation*}
This is a linear differential equation for $\curv \o_0$. So if $\curv\o_0(0)=0$, then $\curv\o_0=0$ for all $t$. Since, for a homotopy $\o$, $\curv\o=\curv\o_0$, we conclude that a homotopy between Maurer--Cartan elements of $\boldsymbol{\mathfrak g}$ is necessarily given by a Maurer--Cartan element of $\O(I,\boldsymbol{\mathfrak g})$, and therefore our notion of homotopy for arbitrary odd elements of a dLie algebra agrees with the notion  used before for the Maurer--Cartan elements.
\end{rem}

Possibly one can obtain further generalizations to $L_{\infty}$-algebras and arbitrary $Q$-manifolds.

In conclusion, we wish to comment on the geometric meaning of the multiplicative integrals in this paper. (This requires a discussion because they are not integrals of $1$-forms, in general.) 

For an ordinary $1$-form, one can integrate it over paths, for a
closed form obtaining a function of the upper limit, which is a
primitive of the form. Similarly, for a  $1$-form $\o\in
\O^1(M,\mathfrak{g})$ with values in a Lie (super)algebra, one can
take multiplicative integrals over paths, which represent the
parallel transport w.r.t. the corresponding connection. In the flat
case this gives a $G$-valued function $g$ on $M$ such that
$\o=\Delta(g)$. The   picture  seems to break down when a $1$-form
is replaced by a form of higher degree, or even more, by an
inhomogeneous or pseudodifferential form\,---\,how  can one integrate something
that is not a $1$-form over a path?   The key of course is that  we
deal with  families of paths, not
individual paths. Even if what we integrate is not a $1$-form, the
integral over paths in a family does make sense, giving a form,
rather than a function, on the space of parameters.  In the flat case, it is just the
supermanifold $M$ parameterizing the upper limit or $M\times M$ for the two limits.
Therefore~\eqref{nonabelpoinc_eq.nabpoinc-primitstar}  and similar
formulas have a geometric interpretation as the parallel transports
over paths in a family  with base $M$ with respect to a flat Quillen
superconnection $d+\o$, where such a parallel transport is described by a
$G$-valued form on $M$.

It comes about naturally to drop the condition of flatness and
consider the  space of all paths on $M$ (of some appropriate class, such as, e.g., piecewise smooth). It is  a ``path groupoid'' $\Paths\, (M)$ (objects are points of $M$, arrows are paths). More precisely, consider
a principal $G$-bundle $P\to M$ over $M$, where $G$ is a
supergroup\,\footnote{\,It may be natural to consider  from the start bundles over
$\Pi TM$.}, endowed with Quillen's superconnection
$\nabla$. Then the  \emph{parallel transport} corresponding to
$\nabla$ is an even {form} on the path groupoid $\Paths\,
(M)$ on $M$  with values
in the transitive (super) Lie groupoid $\Frame(P)$ associated with the principal
bundle $P\to M$, i.e., it is a supermanifold map
\begin{equation}\label{dis.eq.parallel}
   \tau\co \Pi T\Paths\, (M) \to \Frame(P)\,,
\end{equation}
which is a morphism of Lie groupoids \footnote{\,$\Paths\, (M)$ is not exactly a groupoid, because even if we consider paths up to reparametrization and thus achieve the associativity,   there  remain problems with  units and inverses. Still, \eqref{dis.eq.parallel} is a morphism to a genuine Lie groupoid; this leads to the temptation to introduce formally a universal groupoid with a morphism  from $\Paths\, (M)$ to it, as some sort of `non-Abelian singular $1$-chains'.}.

\smallskip
I thank   H.~M.~Khudaverdian and K.~C.~H.~Mackenzie for many inspiring discussions. I am much grateful to J.~D.~Stasheff
for reading  through the first version of the text and providing numerous valuable
comments on style as well as mathematics and for sharing the unpublished manuscript~\cite{schless:stasheff-secret}, and to A.~Lazarev for attracting my attention to  paper~\cite{chuang:lazarev2008}. It is a pleasure to thank them all.


\def\cprime{$'$}

\end{document}